\documentclass[12pt]{amsart}

\usepackage{amssymb}
\usepackage{amsmath}
\usepackage{amsfonts}
\usepackage[all]{xy} % For commutative diagrams
\usepackage{graphicx} % For figures
\usepackage{color} %this is to stop xfig errors.
\newcommand{\disk}{\ensuremath{\mathbb{D}} } % unit disk
\newcommand{\riem}{\Sigma}
\newcommand{\sphere}{\overline{\mathbb{C}}}
\newcommand{\Gr}{\operatorname{Gr}}
\newcommand{\Grfull}{\operatorname{\mathbf{Gr}}}

        \newtheorem{theorem}{Theorem}[section]

        \newtheorem{definition}[theorem]{Definition}

    \newtheorem{remark}[theorem]{Remark}
\usepackage{fullpage}

\title{Dirichlet spaces of domains bounded by quasicircles}
\author{David Radnell}
\address{David Radnell \\ Department of Mathematics and Systems Analysis \\
Aalto University \\
P.O. Box 11100,  FI-00076 Aalto, Finland }
\email{david.radnell@aalto.fi}

\author{Eric Schippers}
\address{Eric Schippers \\ Department of Mathematics \\
University of Manitoba\\
Winnipeg, Manitoba \\  R3T 2N2 \\ Canada}
\email{eric\_schippers@umanitoba.ca}

\author{Wolfgang Staubach}
\address{Wolfgang Staubach\\ Department of Mathematics\\
Uppsala University\\
Box 480\\ 751 06 Uppsala\\ Sweden}
\email{wulf@math.uu.se}
\thanks{D. Radnell acknowledges support from the Academy of Finland. E. Schippers and W. Staubach are grateful for the financial support from the Wenner-Gren Foundations. E. Schippers is also partially supported by the National Sciences and Engineering Research Council of Canada.}

\begin{document}

\begin{abstract}
 Consider a multiply-connected domain $\Sigma$ in the sphere bounded by $n$ non-intersecting quasicircles.
 We characterize the Dirichlet space of $\Sigma$ as an isomorphic image of a direct sum of Dirichlet spaces of the
 disk under a generalized Faber operator.  This Faber operator is constructed using a jump formula for quasicircles and certain spaces of boundary values.

Thereafter, we define a Grunsky operator on direct sums of Dirichlet spaces of the disk, and give
 a second characterization of the Dirichlet space of $\Sigma$ as the graph of the generalized Grunsky
 operator in direct sums of the space $\mathcal{H}^{1/2}(\mathbb{S}^1)$ on the circle.
 This has an interpretation in terms of Fourier decompositions of Dirichlet space functions on the circle.
\end{abstract}

\maketitle

\begin{section}{Introduction}
 Let $\Sigma$ be a domain in $\sphere$ bounded by $n$ non-overlapping quasicircles $\Gamma_k$, $k=1,\ldots,n$.
 Assume that $\Sigma$ contains $\infty$.
 Let $\Omega_k^+$ and $\Omega_k^-$ denote the bounded and unbounded components of the complements of $\Gamma_k$ in the
 sphere $\sphere$, so that $\Sigma = \cap_{k=1}^n \Omega_k^-$.
 Let $f= (f_1,\ldots,f_n)$ be an $n$-tuple of quasiconformally extendible conformal maps
 $f_k:\disk^+ \rightarrow \Omega^+_k$, where $\disk^+ = \{ z\,:\,|z| < 1 \}$.

 Let $\mathcal{D}_\infty(\riem)$ denote the Dirichlet space of holomorphic functions on $\Sigma$ vanishing at $\infty$ and let $\mathcal{D}_\infty(\disk^-)$ denote the Dirichlet space of holomorphic functions on $\disk^- = \{ z\,:\,|z| >1 \} \cup \{\infty \}$ vanishing at $\infty$.
 For Cauchy projections $P(\Omega^-_k)$ from functions on $\Gamma_i$ into $\mathcal{D}_\infty(\Omega_k^-)$
 define
 \begin{align*}
  \mathbf{I}_f : \bigoplus^n \mathcal{D}_\infty(\disk^-) & \longrightarrow \mathcal{D}_\infty(\riem) \\
   (h_1,\ldots,h_n) & \longmapsto \sum_{k=1}^n P(\Omega^-_k) \mathcal{C}_{f^{-1}_k} h_k
 \end{align*}
 where $\mathcal{C}_{f^{-1}_k} h_k = \left. h_k \circ f_k^{-1} \right|_{\Gamma_k}$.  
 We show that $\mathbf{I}_f$ is defined and is an isomorphism.  As an application, we show that Dirichlet space functions have a unique multiple Faber series approximation which converges uniformly on compact subsets of  $\riem$.  

 It is far from obvious that the composition and projection operators are well-defined and bounded,
 since $f_k$ maps into the complement of $\riem$ for each $k$.
 These problems were resolved by E. Schippers and W. Staubach in  \cite{SchippersStaubach_jump,SchippersStaubach_Grunsky_quasicircle} where  a theory of
 harmonic extension and a Plemelj-Sokhotski jump decomposition for quasicircles is developed.
  Since quasicircles are not necessarily	 rectifiable, to define the projection operator, we use a limiting integral along curves approaching $\Gamma$.  In \cite{SchippersStaubach_jump} it was shown that for quasicircles, the integral can be taken over limiting curves approaching $\Gamma$ from either the inside or the outside with the same result.  This is a key theorem which is necessary for the results obtained in the present paper. This in turn relied on the fact that for every Dirichlet-bounded harmonic function on one side of a curve, there is a Dirichlet-bounded harmonic function on the other side with the same boundary values, which we will call the ``reflection''; furthermore this correspondence is bounded. \\

 The isomorphism
 $\mathbf{I}_f$ itself can be thought of as a kind of Faber operator for multiply-connected domains.
 The Faber operator is of general interest in approximation theory \cite{Anderson}, \cite[Chapter VII]{Suetin_monograph}
 since it determines the Faber series approximation of a given function.  
    Different choices of regularity of the boundary and norm of the functions can be made (see the table on page 91 of \cite{Suetin_monograph}).
  The version of the Faber operator $\mathbf{I}_f$ given here is closely related, in the simply connected case, to Faber
  series for Dirichlet space functions \cite{Cavus,SchippersStaubach_Grunsky_quasicircle,Shen_Faber}.
 In this setting existence of Faber series was shown by A. \c{C}avu\c{s} \cite{Cavus}, and existence and uniqueness was proven by Y. Shen \cite{Shen_Faber}.  
  
  When treated as a map between Dirichlet spaces $\mathcal{D}_\infty(\disk^-)$ and $\mathcal{D}(\Omega^-)$, the Faber operator was shown by Schippers and Staubach \cite{SchippersStaubach_Grunsky_quasicircle} to be an isomorphism if and only if the curve is a quasicircle
  \cite{SchippersStaubach_Grunsky_quasicircle}.  This result was first obtained by Shen \cite{Shen_Faber}, with a somewhat different formulation: there, the isomorphism takes $\ell^2$ sequences to Faber series.  Later, H.Y.Wei, M.L. Wang and Y. Hu showed that for rectifiable Jordan curves, the Faber operator on Dirichlet space is an isomorphism if and only if the curve is a rectifiable quasicircle \cite{WeiWangHu}.  They also disproved a conjecture of Anderson that the Faber operator is a bounded isomorphism of Besov spaces $B^p$ for general Jordan curves.  The formulation of Schippers and Staubach as a map between
   Dirichlet spaces for general quasicircles (in terms of the
  composition operator and Cauchy-type projection) required the jump formula on quasicircles as well as the existence of a bounded reflection of Dirichlet-bounded harmonic functions, 
  obtained in \cite{SchippersStaubach_jump,SchippersStaubach_Grunsky_quasicircle}.     
  
  Dirichlet space seems to be the weakest ($L^2$-type) regularity for which this is possible for quasicircles.  Hardy space functions, for example, have only boundary values almost everywhere with respect to harmonic measure.  Since  sets of harmonic measure zero with respect to one side of a Jordan curve need not have harmonic measure zero with respect to the other side, the boundary value problem on the complement will not be well-posed.  Dirichlet space functions have boundary values except on a set of capacity
  zero.  This stronger result makes the problem well-posed.  
  
  For other recent results for Faber operators with different choices of regularity,  see for example Y. E. Y{\i}ld{\i}r{\i}r and R. \c{C}etinta\c{s} \cite{YildirirCetintas}  for Hardy-Orlicz and Smirnov-Orlicz spaces for Dini-smooth curves, or D. Gaier \cite{Gaier} where it is shown that the Faber operator can be unbounded
  for quasicircles with respect to the supremum norm, in contrast to its behaviour with respect to the Dirichlet norm. \\

We also define a generalized Grunsky operator corresponding to the aforementioned $n$-tuple of quasiconformally extendible maps $(f_1,\ldots,f_n)$.
 We show that our generalized Grunsky operator characterizes the Dirichlet space $\mathcal{D}(\riem)$
 of holomorphic functions with bounded Dirichlet energy in the following sense: up to constants, the set
 \begin{equation*} \label{eq:W_definition}
   \mathbf{W} = \left\{ \left(\left. h \circ f_1 \right|_{\mathbb{S}^1},\ldots,
 \left. h \circ f_n \right|_{\mathbb{S}^1} \right) \,:\, h \in \mathcal{D}(\riem) \right\}
 \subseteq \bigoplus^n \mathcal{H}(\mathbb{S}^1),   
 \end{equation*} 
 where $\mathcal{H}(\mathbb{S}^1)$ is the $H^{1/2}$ Sobolev space on $\mathbb{S}^1$, is
 the graph of our Grunsky operator.  This can be thought of as a characterization of the
 possible Fourier series on $\mathbb{S}^1$ obtained by pulling back elements of $\mathcal{D}(\riem)$.
 This characterization relates to the so-called Segal-Wilson Grassmannian, and has applications to two-dimensional conformal field theory and to Teichm\"uller theory, see D. Radnell, E. Schippers and W. Staubach \cite{RSS_contemp,RSS_period_genus_zero}. 
 
 As in the case of Faber operators, the regularity of the curve relates to the properties of the Grunsky operator in the simply connected case.  As is well
  known \cite{Pommerenkebook}, the Grunsky operator (formulated
 in terms of $\ell^2$ sequences) is bounded by a constant strictly less than one, precisely for quasicircles.
 We show that the generalized Grunsky operator is strictly less than one in norm for multiply-connected domains bounded by quasicircles.  It was shown independently by Shen \cite{ShenGrunsky} and L. Takhtajan and L.- P. Teo \cite{TakhtajanTeo_memoirs} that a Jordan curve is a Weil-Petersson class quasicircle if and only if the
 Grunsky operator is Hilbert-Schmidt and bounded by one.  Shen also proved that 
 the Grunsky operator is compact if and only if the curve corresponds to a univalent 
 function with asymptotically conformal quasiconformal extension \cite{Shen_Faber}.
 \\

Here is an outline of the paper.
In Section \ref{se:reflection} we review some facts and results on the Dirichlet spaces, the harmonic reflection for quasicircles and the projection operators. In Section \ref{se:faber_grunsky} we explain the relation of the projection operators to the Grunsky operator given in the papers \cite{SchippersStaubach_jump,SchippersStaubach_Grunsky_quasicircle}, and
  formulate them in a way suitable for this paper.  We then prove
  the first main result that the Faber operator is an isomorphism 
  of Dirichlet spaces defined in \cite{SchippersStaubach_Grunsky_quasicircle}
  to multiply-connected domains bounded by quasicircles.   We apply this to show that elements of the 
  Dirichlet space have unique multiple Faber series, which converge uniformly on compact sets.  In Section \ref{se:Dirichlet_Grunsky} we define the generalized
  Grunsky operator, derive some integral expressions for its blocks, and show that it is bounded
  by some $k <1$ for multiply-connected domains bounded by quasicircles.  Finally we 
  show that $\mathbf{W}$ is the graph of the Grunsky operator.  

\end{section}

\begin{section}{Reflection of Dirichlet spaces in quasicircles} \label{se:reflection}
\begin{subsection}{Notation}
 Let $\disk^+ = \{ z\,:\,|z| <1 \}$, $\disk^- = \{ z\,:\,|z| >1 \} \cup \{ \infty \}$, $\mathbb{S}^1 = \{ z\,:\,|z| = 1 \}$, and $\sphere$ be the Riemann sphere $\mathbb{C} \cup \{ \infty \}$.  Complements are taken in the Riemann
 sphere.  Throughout the paper $\Gamma$ will be either a Jordan curve or specifically a quasicircle, not passing through $\infty$,
 and we denote the bounded component of its complement by $\Omega^+$ and the unbounded component
 by $\Omega^-$.  We will shorten this by saying $\Gamma$ bounds $\Omega^\pm$.

  For domains $\Omega_i$ in $\sphere$, $i=1,2$,
we say that $f:\Omega_1 \rightarrow \Omega_2$ is conformal if it is a meromorphic bijection. In particular,
$f$ has at most a simple pole.

\end{subsection}
\begin{subsection}{Dirichlet spaces}
 Let $\Omega$ be a Jordan domain in $\sphere$.  Throughout the paper, $\infty$ will not
 lie on the boundary of $\Omega$.  Using conformal invariance, the definitions below extend in an obvious way to domains
 whose boundaries contain $\infty$, although we will not require this.

 We define the harmonic Dirichlet space
\begin{equation}
 \mathcal{D}_{\mathrm{harm}}(\Omega)  = \left\{ h \colon \Omega \rightarrow \mathbb{C} \,:\,h \ \mbox{harmonic and} \ \iint_{\Omega} \left| \frac{\partial h}{\partial z}\right|^2 \,dA +
   \iint_{\Omega} \left|\frac{\partial h}{\partial \bar{z}} \right|^2 \,dA <\infty \right\}.
\end{equation}
If $\infty \in \Omega$, then the statement that $h$ is harmonic is taken to mean that $h$ is harmonic on $\Omega \backslash \{\infty \}$
and $h(1/z)$ is harmonic on an open set containing $0$. In a similar way, we define the notion of holomorphicity at $\infty$.  We define a semi-norm on $\mathcal{D}_{\mathrm{harm}}(\Omega)$ by
\begin{equation}\label{defn:norm in Dharm}
\| h \|^{2}_{\mathcal{D}_{\mathrm{harm}}(\Omega)}  = \iint_{\Omega}  \left| \frac{\partial h}{\partial z}\right|^2 \,dA
+ \iint_{\Omega} \left| \frac{\partial h}{\partial \bar{z}}\right|^2 \,dA.
 \end{equation}
This semi-norm is invariant under composition by M\"obius transformations; that is, for any M\"obius transformation $T$ one has
\begin{equation} \label{eq:Mobius_invariance}
 \| h \|_{\mathcal{D}_{\mathrm{harm}}(T(\Omega))} = \| h \circ T \|_{\mathcal{D}_{\mathrm{harm}}(\Omega)}.
\end{equation}
In particular, we have that $h \in \mathcal{D}_{\mathrm{harm}}(T(\Omega))$ if and only if $h \circ T \in \mathcal{D}_{\mathrm{harm}}(\Omega)$.

We denote the holomorphic Dirichlet spaces by
\[  \mathcal{D}(\Omega) = \{ h \in \mathcal{D}_{\mathrm{harm}}(\Omega)\,:\, h \text{ holomorphic} \}.  \]
The antiholomorphic Dirichlet space consisting of $\overline{h}$ such that $h \in \mathcal{D}(\Omega)$
will be denoted $\overline{\mathcal{D}(\Omega)}$.
The restriction of the semi-norm to $\mathcal{D}(\Omega)$ is given by
\[  \| h \|_{\mathcal{D}(\Omega)}^2 = \iint_{\Omega} |h'|^2 \,dA. \]

When $\Omega$ contains $\infty$,
we will also consider the subspace of those functions in $\mathcal{D}(\Omega)$ vanishing at $\infty$.
We will denote this by
{\[  \mathcal{D}_\infty(\Omega)= \{ h \in \mathcal{D}(\Omega) \,: h(\infty) = 0 \}.  \]}
On this space the seminorm is a norm.

Finally, we observe the following elementary conformal invariance.
 Let $\Omega_1$ and $\Omega_2$ be Jordan domains not containing $\infty$ in
 their boundaries.  Let $f:\Omega_1 \rightarrow \Omega_2$ be a conformal map.  The composition operator
 \begin{align*}
  \mathcal{C}_f: \mathcal{D}_{\mathrm{harm}}(\Omega_2) & \longrightarrow \mathcal{D}_{\mathrm{harm}}(\Omega_1) \\
  h & \longmapsto h \circ f
 \end{align*}
 is a linear isomorphism such that
 \begin{equation} \label{eq:composition_invariant}
 \| \mathcal{C}_f h \|_{\mathcal{D}_{\mathrm{harm}}(\Omega_1)} = \|   h \|_{\mathcal{D}_{\mathrm{harm}}(\Omega_2)}.
 \end{equation}
The restriction of the composition operator to the holomorphic Dirichlet spaces is defined in a similar way.

\end{subsection}

\begin{subsection}{Boundary values of Dirichlet space and harmonic reflection}

The Dirichlet spaces have boundary values in a certain conformally invariant sense originating
with Osborn \cite{Osborn}.  We summarize
  the results here; a full exposition and proofs can be found in \cite{SchippersStaubach_jump}.\\

Let $\Gamma$ be a
Jordan curve, not containing $\infty$, bounding a Jordan domain $\Omega$.  Here $\Omega$ could be either the bounded or unbounded component of the complement of $\Gamma$.
Fix any point $p \in \Omega$.
For any point $q \in \Gamma$ let $\gamma_{p,q}$ be the
hyperbolic geodesic arc connecting $p$ to $q$. For any
 $h \in \mathcal{D}_{\mathrm{harm}}(\Omega)$, the limiting values
  \[  H(q) = \lim_{z \rightarrow q} h(z), \quad z \in \gamma_{p,q}  \]
  exist for all $q$, except possibly on a set $K$ which is the image of a Borel set $I \subseteq \mathbb{S}^1$ of zero logarithmic capacity under a conformal map $f:\disk^+ \rightarrow \Omega$.  Since disk automorphisms take sets of logarithmic capacity zero in $\mathbb{S}^1$ to sets of logarithmic capacity zero, this is independent of the choice of $p$ and $f$.
  Let $\mathcal{H}(\partial \Omega)$ temporarily denote the set of functions on $\Gamma$ obtained in this way, where two
  functions are identified if they agree except on a set of type $K$.  Any $h \in \mathcal{H}(\partial \Omega)$ has a unique extension to $\mathcal{D}(\Omega)$ whose boundary values
  agree with $h$ except on a set of type $K$.

In fact $\mathcal{H}(\partial \Omega)$ is the set of boundary values of $\mathcal{D}_{\text{harm}}(\Omega)$
on the ideal boundary of $\Omega$, which can be identified with $\Gamma$ by the Carath\'eodory theorem.
We call the set of functions obtained as boundary values of $\mathcal{D}_{\text{harm}}(\Omega)$ in this way
the  {\it Douglas-Osborn space.}  We also call these boundary values in the sense of Osborn.  \\

Now assume that $\Gamma$ is a quasicircle. Let us denote the set of boundary values of $\mathcal{D}_{\text{harm}}(\Omega^\pm)$, by $\mathcal{H}_\pm(\Gamma).$
It is not clear whether $\mathcal{H}_+(\Gamma)$ and $\mathcal{H}_-(\Gamma)$ are the same
in general.  However,   on
quasidisks we have by a result in \cite[Theorem 2.14]{SchippersStaubach_jump} that
\[  \mathcal{H}_+(\Gamma) = \mathcal{H}_-(\Gamma)  \]
in the following sense.  Fix conformal maps $f:\disk^+ \rightarrow \Omega^+$ and $g:\disk^- \rightarrow \Omega^-$.  For any element $h_+ \in \mathcal{D}(\Omega^+)$, there is a unique $h_- \in \mathcal{D}(\Omega^-)$ whose boundary values agree with those of $h_+$ except on a set $K$ such that
$K=f(I)=g(J)$ where $I \subseteq \mathbb{S}^1$ and $J \subseteq \mathbb{S}^1$ are sets of logarithmic capacity zero.   Thus for quasicircles $\Gamma$ not containing $\infty$ we set
\[  \mathcal{H}(\Gamma) = \mathcal{H}_+(\Gamma) = \mathcal{H}_-(\Gamma)  \]
and make the following definition.
\begin{definition}  Let $\Gamma$ be a quasicircle not containing $\infty$.   For $h \in \mathcal{H}(\Gamma)$ let $\mathbf{E}_\pm h$ denote the unique elements of $\mathcal{D}(\Omega^\pm)$
whose boundary values in the sense of Osborn agree with those of $h$.
\end{definition}

\begin{remark}  In fact, the condition that $\mathcal{H}_+(\Gamma)=\mathcal{H}_-(\Gamma)$ characterizes quasidisks in a
 certain sense, see \cite[Theorem 2.14]{SchippersStaubach_jump}.
\end{remark}

Using $\mathbf{E}_\pm$, one can define semi-norms on $\mathcal{H}_\pm(\Gamma)$ by
\begin{equation}\label{eq: seminorms on h plusminus}
\| h \|_{\mathcal{H}_\pm(\Gamma)} = \| \mathbf{E}_\pm h \|_{\mathcal{D}_{\mathrm{harm}}(\Omega^\pm)},
\end{equation}
and one also has that there is a uniform constant $C$, depending
only on $\Gamma$, such that for any $h \in \mathcal{H}(\Gamma)$,
\[  \frac{1}{C} \| \mathbf{E}_+h \|_{\mathcal{D}(\Omega^+)_{\text{harm}}}
   \leq \| \mathbf{E}_- h \|_{\mathcal{D}(\Omega^-)_{\text{harm}}} \leq C  \| \mathbf{E}_+h \|_{\mathcal{D}(\Omega^+)_{\text{harm}}},  \]
see \cite[Equation (2.5)]{SchippersStaubach_jump}.

Observe that $\mathcal{H}(\Gamma)$ is conformally invariant in the following sense \cite[Theorem 2.21]{SchippersStaubach_jump}:\\

Let $\Gamma_1$ and $\Gamma_2$ be quasicircles not containing $\infty$, bounding domains $\Omega_1$ and $\Omega_2$.
A conformal map $f:\Omega_1 \rightarrow \Omega_2$ must have a quasiconformal extension to $\sphere$ since $\Gamma_i$
are quasicircles \cite{Lehto}. Moreover, quasiconformal maps preserve sets of zero logarithmic capacity.  Thus, we have a well-defined composition operator
\begin{align*}
 \mathcal{C}_f: \mathcal{H}(\Gamma_2) & \longrightarrow \mathcal{H}(\Gamma_1) \\
 h & \longmapsto h \circ f.
\end{align*}

\begin{theorem}  \label{th:Osborn_oneside_conf_inv}
 Let $\Gamma_1$ and $\Gamma_2$ be quasicircles not containing $\infty$, bounding
 domains $\Omega_1^\pm$ and $\Omega_2^\pm$.  If $f^\pm:\Omega_1^\pm \rightarrow \Omega_2^\pm$
  are conformal maps, then $\mathcal{C}_{f^\pm}$ are linear isomorphisms such that
  \[  \| \mathcal{C}_{f^\pm} h \|_{\mathcal{H}_{\pm}(\Gamma_1)} =  \|  h \|_{\mathcal{H}_{\pm}(\Gamma_2)}.
    \]
\end{theorem}

Throughout the paper, we will not specify which semi-norm is placed on $\mathcal{H}(\Gamma)$.  From this point of view, Theorem \ref{th:Osborn_oneside_conf_inv} says only that the composition operator is bounded with respect to the seminorm.

\begin{remark}
In the case that $\Gamma=\mathbb{S}^1$ we have that
\[  \mathcal{H}_+(\mathbb{S}^1) = \mathcal{H}_-(\mathbb{S}^1)=  H^{1/2}(\mathbb{S}^1)= \mathcal{H}(\mathbb{S}^1) \]
where $H^{1/2}(\mathbb{S}^1)$ denotes the set of $L^2$ functions on $\mathbb{S}^1$ whose Fourier series
$h(e^{i\theta}) = \sum_{n=-\infty}^\infty h_n e^{i n \theta}$ satisfy
\[  \sum_{n=-\infty}^\infty |n||h_n|^2 <\infty.  \]
It is also well-known that functions $h \in \mathcal{H}(\mathbb{S}^1)$ have Fourier series which converge to $h$
almost everywhere, in fact except on a set of outer logarithmic capacity zero \cite{El-Fallah_etal_primer}.  We can assume the set is a Borel set of logarithmic capacity zero.
\end{remark}

\end{subsection}

\begin{subsection}{Decompositions of Dirichlet spaces on quasidisks and the projection operators}
 For a fixed $p\in \Omega^+$ we define the holomorphic Dirichlet space on $\Omega^+$ by
\[  \mathcal{D}_p(\Omega^+) = \left\{ h:\Omega^+ \rightarrow \mathbb{C} \,: \,  h \text{ holomorphic}, \ h(p)=0, \text{ and }  \iint_{\Omega^+} |h'|^2 \,dA < \infty \right\}  \]
 and similarly on $\Omega^-$ by
\[   \mathcal{D}_\infty(\Omega^-) = \left\{ h:\Omega^- \rightarrow \mathbb{C} \,: \,  h \text{ holomorphic}, \ h(\infty)=0, \text{ and }  \iint_{\Omega^-} |h'|^2 \,dA < \infty \right\}.  \]
Since the domains $\Omega^{\pm}$ are simply-connected it is immediate that
\[  \mathcal{D}_{\mathrm{harm}}(\Omega^+) = \mathcal{D}_p(\Omega^+) \oplus  \overline{\mathcal{D}_p(\Omega^+)} \oplus \mathbb{C} \]
and
\[  \mathcal{D}_{\mathrm{harm}}(\Omega^-) = \mathcal{D}_\infty(\Omega^-) \oplus
 \overline{\mathcal{D}_\infty(\Omega^-)} \oplus \mathbb{C} \]
where $\overline{\mathcal{D}_p(\Omega^+)}$ denotes the anti-holomorphic functions on $\Omega^+$ vanishing at $p$ and
with finite Dirichlet energy (equivalently, conjugates of elements of $\mathcal{D}_p(\Omega^+)$
and similarly for $\Omega^-$).  The $\mathbb{C}$ in the decomposition corresponds to the
constants $h(p)$ and $h(\infty)$ respectively.

Moreover, we have a reflection
 map
 \begin{align*}
  \mathbf{R}: \mathcal{D}_\infty(\disk^-) & \longrightarrow \overline{\mathcal{D}_0(\disk^+)}\\
  h(z) & \longmapsto h(1/\bar{z}).
 \end{align*}
 and similarly $\mathbf{R}:\mathcal{D}_0(\disk^+) \rightarrow \overline{\mathcal{D}_\infty(\disk^-)}$.
 It is easily seen that $\mathbf{R}$ is an isomorphism.   Furthermore $\mathbf{R}$
 preserves the boundary values of $h$.  Thus the decomposition of $\mathcal{D}_{\mathrm{harm}}(\Omega^+)$ also
 agrees with the decomposition $\mathcal{H}(\mathbb{S}^1) = \mathcal{D}_0(\disk^+) \oplus \mathcal{D}_\infty(\disk^-) \oplus \mathbb{C}$.

Now we define the Cauchy projections. In the following definition, and the rest of the paper,
we denote  by $\gamma_r$ the curve $|z|=r$ oriented counterclockwise.
\begin{definition}
 Let $\Gamma$ be a quasicircle not containing $\infty$.  Let $\Omega^+$ and $\Omega^-$ be the bounded and
 unbounded components of the complement respectively.  Let $f:\disk^+ \rightarrow \Omega^+$ be conformal.

 We define, for $z \in \Omega^-$,
 \begin{align*}
  \mathrm{P}_\infty(\Omega^-):\mathcal{H}(\Gamma) & \longrightarrow \mathcal{D}_\infty(\Omega^-) \\
  h & \longmapsto - \lim_{r \nearrow 1} \frac{1}{2 \pi i} \int_{f(\gamma_r)} \frac{\mathbf{E}_+ h (\zeta)}{\zeta - z}\, d\zeta
 \end{align*}
 and for $z \in \Omega^+$,
 \begin{align*}
  \mathrm{P}_p(\Omega^+):\mathcal{H}(\Gamma) & \longrightarrow \mathcal{D}_p(\Omega^+) \\
  h & \longmapsto \lim_{r \nearrow 1} \frac{1}{2 \pi i} \int_{f(\gamma_r)} \frac{(z-p)\,   \mathbf{E}_+h(\zeta)}{(\zeta-p)(\zeta - z)}\, d\zeta.
 \end{align*}
\end{definition}

These projections are well-defined and bounded.
\begin{theorem} \label{th:all_about_projection}
 Let $\Gamma$ be a quasicircle not containing $\infty$.  Let $\Omega^+$ and $\Omega^-$ be the bounded and
 unbounded components of the complement respectively.   Then
 \begin{enumerate}
  \item $\mathrm{P}_p(\Omega^+)$ and $\mathrm{P}_\infty(\Omega^-)$ are independent of the
  choice of $f$,
  \item $\mathrm{P}_p(\Omega^+)$ is a bounded map into $\mathcal{D}_p(\Omega^+)$, and
  \item $\mathrm{P}_\infty(\Omega^-)$ is bounded map into $\mathcal{D}_\infty(\Omega^-)$.
 \end{enumerate}
 Furthermore, for any  conformal map $g:\disk^- \rightarrow \Omega^-$ we have
 the alternate integral formulas
 \[  \mathrm{P}_p(\Omega^+)h (z) = \lim_{r \searrow 1} \frac{1}{2 \pi i} \int_{g(\gamma_r)} \frac{(z-p)\, \mathbf{E}_- h(\zeta)}
     {(\zeta-p)(\zeta - z)}\, d\zeta, \quad \text{for }  z \in \Omega^+, \]
 and
 \[  \mathrm{P}_\infty(\Omega^-) h (z) = - \lim_{r \searrow 1} \frac{1}{2 \pi i} \int_{g(\gamma_r)} \frac{\mathbf{E}_- h(\zeta)}
     {\zeta - z}\, d\zeta, \quad \text{for }  z \in \Omega^-. \]
\end{theorem}
\begin{proof}Observe that
 the projection from the full Dirichlet space $\mathcal{D}(\Omega^+)$ to $\mathcal{D}_p(\Omega^+)$
 given by $h \mapsto h - h(p)$ is bounded with respect to the norm $\|h\|^2_{\mathcal{D}(\Omega^+)} = \|h\|_{\mathcal{D}_p(\Omega^+)}^2 + |h(p)|^2$.  The integral kernel of this projection is
 \[  \frac{1}{2\pi i} \frac{(z-p)}{(\zeta-p)(\zeta-z)}.  \]
 The claim thus reduces to \cite[Theorems 3.2, 3.5]{SchippersStaubach_jump}.
\end{proof}

We remark that on the disk, the integral can be taken over the boundary: that is,
\[  \mathrm{P}_p(\disk^+)h(z) = \frac{1}{2 \pi i} \int_{\mathbb{S}^1} \frac{(z-p)\, \mathbf{E}_+ h(\zeta)}{(\zeta-p)(\zeta - z)}\, d\zeta  \]
and similarly for $\mathrm{P}_\infty(\disk^-)$.  In fact this holds for any Weil-Petersson
class quasidisk \cite{RSS_WPjump}.
\end{subsection}
\end{section}

\begin{section}{The Faber isomorphism and Grunsky operators}\label{se:faber_grunsky}
\begin{subsection}{Faber isomorphism and Grunsky operator}
 Next, we define two operators. We notationally suppress a trace to the boundary in both definitions for simplicity.
 \begin{definition}
  Let $\Gamma$ be a quasicircle not containing $\infty$, with bounded and unbounded
  components $\Omega^+$ and $\Omega^-$ respectively. Let $f: \disk^+ \to \Omega^+$ be a conformal map.  Define $I_f$ by
  \begin{align*}
   \mathrm{I}_f: \mathcal{D}_\infty(\disk^-) & \longrightarrow \mathcal{D}_\infty(\Omega^-) \\
   h & \longmapsto \mathrm{P}_\infty(\Omega^-) \mathcal{C}_{f^{-1}} h,
  \end{align*}

  and define the Grunsky operator by
  \begin{align*}
   \Gr_f: \mathcal{D}_\infty(\disk^-) & \longrightarrow \mathcal{D}_0(\disk^+) \\
   h & \longmapsto \mathrm{P}_0(\disk^+) \mathcal{C}_f \mathrm{I}_f h.
  \end{align*}
 \end{definition}
As it is clear from the notations and the definitions, the maps $\mathrm{I}_f$ and $\mathrm{Gr}_f$ depend on $f$.
 \begin{remark}  Our definition of the Grunsky operator differs slightly from that in
  \cite{SchippersStaubach_Grunsky_quasicircle}. In that paper, the projection did not annihilate the constant as it does here.
 \end{remark}

  In \cite[Theorems 3.7, 3.14]{SchippersStaubach_Grunsky_quasicircle} we showed the following facts.
  \begin{theorem} \label{th:If_isomorphism_Grunsky_bounded}  $\Gr_f$ is a bounded operator and
   $\mathrm{I}_f$ is a bounded isomorphism.
  \end{theorem}

 We register useful alternate expressions  for $\mathrm{I}_f$ and $\Gr_f$.
 For $h \in \mathcal{D}_\infty(\disk^-)$,
  \begin{equation}  \label{eq:If_alternate_using_r}
   \mathrm{I}_f h(z) =  - \lim_{r \nearrow 1} \frac{1}{2 \pi i} \int_{f(\gamma_r)}
    \frac{(\mathbf{R} h) \circ f^{-1}(\zeta)}{\zeta - z} \, d\zeta.
  \end{equation}
  This follows immediately from the observation that $(\mathbf{R} h) \circ f^{-1}$
  is the harmonic extension of $\mathcal{C}_{f^{-1}} h$ to $\Omega^+$.

  \begin{remark}  It can be shown that if $h$ is a polynomial in $1/z$, then $\mathbf{R}h$ can be replaced by $h$ in the
  above expression \cite[Theorem 3.4]{SchippersStaubach_Grunsky_quasicircle}.
   In particular, $\mathrm{I}_f (z^{-n})$ is the $n$th Faber
  polynomial of the domain $\Omega^-$.  Although we will not need this result, it is closely related to the fact that $\mathrm{I}_f$ is an isomorphism.
  \end{remark}

 \begin{theorem}   \label{th:Grunsky_integral} The Grunsky operator
  has the integral expression
  \begin{equation*}
   \left[\Gr_f h \right](z)=  \lim_{r \nearrow 1} \frac{1}{2\pi i} \int_{\gamma_r}
    \left[ \frac{z}{\zeta(\zeta - z)} - \frac{f'(\zeta)}{f(\zeta)-f(z)} + \frac{f'(\zeta)}{f(\zeta)-f(0)} \right]   (\mathbf{R} h)(\zeta) \, d\zeta.
  \end{equation*}
 \end{theorem}
 \begin{proof}
  Let $P(\disk^+)$ denote the projection onto $\mathcal{D}(\disk^+)$, that is, the full Dirichlet space
  without the condition that the functions be constant at $0$.  In \cite[Theorem 3.9]{SchippersStaubach_Grunsky_quasicircle},
  we showed that the operator $\breve{\Gr}_f = P(\disk^+) \mathcal{C}_f \mathrm{I}_f$ has the integral expression
  \[   \breve{\Gr}_f h (z)=  \lim_{r \nearrow 1} \frac{1}{2\pi i} \int_{\gamma_r}
    \left[ \frac{1}{\zeta - z} - \frac{f'(\zeta)}{f(\zeta)-f(z)} \right]   (\mathbf{R} h)(\zeta) \, d\zeta.   \]
  The claim now follows from the facts that
  $\Gr_f h (z)= \breve{\Gr}_f h (z) -\breve{\Gr}_f h (0)$
  and
  \[ \breve{\Gr}_f h (0) =   \frac{1}{2 \pi i} \int_{\gamma_r} \left[ \frac{1}{\zeta}  -
     \frac{f'(\zeta)}{f(\zeta)-f(0)} \right](\mathbf{R} h)(\zeta) d\zeta.  \]
 \end{proof}
\end{subsection}

\begin{subsection}{Isomorphism onto Dirichlet spaces of multiply-connected domains}
 In this section we generalize the isomorphism $\mathbf{I}_f$ to multiply-connected domains.
 Let $\riem$ be a domain in $\sphere$, bordered by non-intersecting quasicircles $\Gamma_i$, $i=1,\ldots,n$,
 such that $\infty \in \riem$.  Let $\Omega_i^+$ and $\Omega_i^-$ denote the bounded and unbounded components
 of $\sphere \backslash \Gamma_i$ respectively.
 Let
 \[  \mathcal{D}_\infty(\Sigma) = \left\{ h:\Sigma \rightarrow \mathbb{C}  \,:\,
 h  \text{ holomorphic, } h(\infty)=0, \text{ and }   \iint_\Sigma | h'|^2 \,dA <\infty  \right\}  \]
  denote the Dirichlet space of $\Sigma$.
 Let
 \[  \mathcal{D}_\infty(\disk^-)^n = \mathcal{D}_\infty(\disk^-) \oplus \cdots \oplus \mathcal{D}_{\infty}(\disk^-)  \]
 where there are $n$ summands on the right hand side.
 Given an $n$-tuple $f=(f_1,\ldots,f_n)$ of conformal maps $f_i:\disk^+ \rightarrow \Omega^{+}_i$,
 we define an isomorphism
 \begin{equation}
 \begin{aligned} \label{eq:definition}
  \mathbf{I}_f:\mathcal{D}_\infty(\disk^-)^n &
     \longrightarrow \mathcal{D}_\infty(\Sigma) \\
 (g_1,\ldots,g_n) & \longmapsto \sum_{i=1}^n \left. \mathrm{I}_{f_i}  g_i \right|_{\Sigma}.
 \end{aligned}
\end{equation}

 To show that this is indeed a bounded isomorphism, we require the following theorems.
 \begin{theorem}  \label{th:multiply_connected_restriction}
  Let $\Sigma$ be a multiply-connected domain in $\sphere$ containing $\infty$ bounded
  by non-intersecting quasicircles $\Gamma_i$, $i=1,\ldots,n$.   Any $h \in \mathcal{D}(\Sigma)$ has boundary values
 $\mathcal{H}(\Gamma_i)$  in the sense of Osborn for all $i=1,\ldots,n$.
 \end{theorem}
 \begin{remark} \label{re:Osborn_multiple_meaning} For a multiply-connected domain of the type in $\mathrm{Theorem}$ $3.6$, by ``boundary values in the sense of
 Osborn'' we mean the following.  Let $h \in \mathcal{D}(\riem)$.   Choose one boundary curve $\Gamma_i$ and let $\Omega_i^-$ be the component of the complement containing $\riem$ as above.  For any fixed point $p \in \Omega_i^-$, let $\gamma_{p,q}$ be the hyperbolic geodesic ray beginning at $p$ and terminating at $q \in \Gamma_i$.  The limit of $h$ exists as $z \rightarrow q$ along $\gamma_{p,q}$ for all $q$ except possibly along a set $K$ which is the image of a Borel set $I$ of logarithmic capacity zero in $\mathbb{S}^1$ under a conformal map $f:\disk^+ \rightarrow \Omega_i^-$.  Note that the hyperbolic geodesic must eventually lie entirely in $\riem$ so this makes sense.  Furthermore, the boundary values in this sense are in $\mathcal{H}(\Gamma_i)$; i.e. there is a harmonic function $H$ in $\mathcal{D}(\Omega_i^-)$ whose boundary values are the same as those of $h$ up to a set of logarithmic capacity zero.
 \end{remark}
% \begin{proof} Let $\Omega_i^+$ and $\Omega_i^-$ be the bounded and unbounded components
%  respectively of $\Gamma_i$.
%  Fix a boundary curve $\Gamma_i$ and let $g_i:\disk^- \rightarrow \Omega_i^-$ be a conformal
%  map such that $g_i(\infty)=\infty$.  Since $h \in \mathcal{D}(\riem)$, $h \circ g_i \in \mathcal{D}(A)$
%  for some annulus $A = \{ z\,:\, 1 <|z|<r \}$.  Thus $h \circ g_i \in H^1(A)$
%  (that is, the Sobolev space of $L^2$ functions with first order derivatives in $L^2$), since $A$ is a domain that has the so called cone property, see e.g. the Corollary on page 21 in \cite{Mazya}.
%  Therefore $h \circ g_i$ has a trace (say, $H$) to the inner boundary which is in $H^{1/2}(\mathbb{S}^1) = \mathcal{H}(\mathbb{S}^1)$
%  (see e.g. \cite{ChazPir} Corollary 4.5).  By Theorem \ref{th:Osborn_oneside_conf_inv},
%   $H \circ g_i^{-1} \in \mathcal{H}(\Gamma_i)$ and extends $h$.
% \end{proof}
 \begin{proof}
  It is enough to prove this in the case that $\Gamma_i = \mathbb{S}^1$ and $\riem$ is an annulus $A_r = \{ z\,:\, r < |z| <1  \}$ for $0< r <1$.  To see this, choose a  boundary curve $\Gamma_i$ and a conformal map $f:\disk^+ \rightarrow \Omega_i^-$.  Choose $r$ such that $f(A_r) \subset \riem$.  Let $H$ be the function in $\mathcal{D}(\disk^+)$ whose boundary values agree with those of $h \circ f$, in the sense of Remark \ref{re:Osborn_multiple_meaning}.  Thus $H \circ f^{-1}$ is in $\mathcal{D}(\Omega_i^-)$ and has boundary values in $\mathcal{H}(\Gamma_i)$ agreeing with $h$.

 Restricting therefore to the case of an annulus $A_r$, assume that $h \in \mathcal{D}(A_r)$. Let $\gamma$ be the curve $|z|=s$ traced counterclockwise for some $s \in (r,1)$.  Let
 \[ G_\pm(z) = \pm \frac{1}{2\pi i} \int_\gamma \frac{h(\zeta)}{\zeta - z} \, d\zeta \]
 where the sign is chosen according to whether $z$ is in the bounded or unbounded component of $\gamma$ respectively.  Since $h$ is holomorphic in $A_r$, by deforming the curve we see that $G_\pm$ are holomorphic in $\disk^\pm$ respectively. By the classical Plemelj-Sokhotski jump relation, on $\gamma$ we have
 $h = G_+ + G_-$ (note that $h$, $G_+$ and $G_-$ are all holomorphic on a neighbourhood of  $\gamma$ so equality holds literally and not just in a limiting sense).  Thus $h = G_+ + G_-$ identically on $A_r$.

 Since $G_-$ is holomorphic in a neighbourhood of $\disk^-$, it is in $\mathcal{D}(\disk^-)$, and since it is continuous on $\mathbb{S}^1$ the limiting values exist everywhere in the sense of Osborn from within $A_r$. Fixing a $t \in (r,1)$, since $G_+ = h - G_-$
 and $h$ and $G_-$ are both in $\mathcal{D}(A_t)$, it follows that $G_+$ is in $\mathcal{D}(A_t)$ and therefore $G_+ \in \mathcal{D}(\disk^+)$.

 Since $h = G_+ + G_-$ on $A_r$ and both $G_+$ and $G_-$ have boundary values in $\mathcal{H}(\mathbb{S}^1)$ in the sense of Osborn, this completes the proof.
 \end{proof}
 \begin{theorem} \label{th:decomposition}
  Let $\Sigma$ be a multiply-connected domain in $\sphere$ containing $\infty$ bounded
  by non-intersecting quasicircles $\Gamma_i$, $i=1,\ldots,n$.  Let $\Omega_i^+$ and $\Omega_i^-$ be the bounded and unbounded components
  respectively of $\Gamma_i$.  Every element $h \in \mathcal{D}_\infty(\Sigma)$ has a unique representation
  \[  h = h_1 + \cdots + h_n, \quad h_i \in \mathcal{D}_\infty(\Omega_i^-),  \]
  where $h_i = \mathrm{P}_\infty(\Omega_i^-) h$ $($that is, we apply $P_\infty(\Omega_{i}^{-})$ to the boundary values of $h$$)$.
 \end{theorem}
 \begin{proof}
   Let $g_i:\disk^- \rightarrow \Omega_i^-$
  be conformal maps such that $g_i(\infty)=\infty$.  For any fixed $z \in \Sigma$, there is   an $R >1$ close enough to $1$ so that $z$ is in the
  unbounded component of the complement of $g_i(\gamma_R)$ for all $i=1,\ldots,n$, and furthermore that
  $g_i(\gamma_r)$ is in $\Sigma$ for $1<r \leq R$.
  Define
  \[  H_i(z) =  - \lim_{r \searrow 1} \frac{1}{2 \pi i} \int_{g_i(\gamma_r) }
   \frac{h(\zeta)}{\zeta - z} \,d\zeta  \]
  where  $z \in \Omega_i^-$, so that $H_i$ is holomorphic in $\Omega_i^-$.   Observe that the integral is eventually independent of $r$.

 We first claim that $H_i = \mathrm{P}_\infty(\Omega_i^-) h$.  First, observe that by
 the Cauchy integral formula $h=H_1 + \cdots + H_n$ in $\riem$. For
 any fixed $i$, $h-H_i = \sum_{k \neq i} H_k$ extends to a holomorphic function
 in $\Omega_i^+$.  This is because every term on the right hand side is a holomorphic
 function in $\Omega_j^-$ for some $j \neq i$, and thus is holomorphic on $\Omega_i^+$
 since $\Omega_i^+ \subset \Omega_j^-$ for all $j \neq i$.  Also note that
 since the closure of $\Omega_i^+$ is compactly contained in $\Omega_j^-$, the
 extension of $h-H_i$ is in $\mathcal{D}(\Omega_i^+)$.  In particular, the boundary values of $h-H_i$ are in $\mathcal{H}(\Gamma_i)$.  By Theorem \ref{th:multiply_connected_restriction}, $h$ has a restriction to the boundary in $\mathcal{H}(\Gamma_i)$.  Thus $H_i = h - (h-H_i)$ has a restriction to the boundary in
 $\mathcal{H}(\Gamma_i)$.

 Thus we may apply $\mathrm{P}_\infty(\Omega_i^-)$ to the boundary values of both $h$ and $H_i$.  Since $h-H_i$ extends to a function in $\mathcal{D}(\Omega_i^+)$, by \cite[Theorem 3.10]{SchippersStaubach_jump}, $\mathrm{P}_\infty(\Omega_i^-) [h - H_i ] =0$ (where by this we mean of course the application of the projection to the boundary values of $h-H_i$).  Again applying \cite[Theorem 3.10]{SchippersStaubach_jump} we obtain that
$\mathrm{P}(\Omega_i^-) H_i = H_i$, which proves that $\mathrm{P}(\Omega_i^-) h
= H_i$ as claimed.

  Now choose $s$ such that $1<s<R$.
  By the Cauchy integral formula,
  since $h$ is holomorphic on $\riem$,
  \begin{align*}
    h(z) & = - \sum_{i=1}^n \frac{1}{2 \pi i} \int_{g_i(\gamma_s)} \frac{h(\zeta)}{\zeta - z} d\zeta \\
    & = - \sum_{i=1}^n   \lim_{r \searrow 1} \frac{1}{2 \pi i} \int_{g_i(\gamma_r)} \frac{h(\zeta)}{\zeta - z} d\zeta \\
    & =  \sum_{i=1}^n \left[ \mathrm{P}_\infty(\Omega_i^-) h \right](z)
  \end{align*}
  where in the last step we have used Theorem \ref{th:all_about_projection}. Now
  defining $h_i = \mathrm{P}_\infty(\Omega_i^-) h$, completes the proof of the existence of the representation.

  To show that the representation is unique, let $u_1 + \cdots + u_n = h$ be another representation
  with $u_i \in \mathcal{D}_\infty(\Omega_i^-)$ for each $i$.  We then have that
  \[  - h_1 + u_1 = h_2 - u_2 + \cdots + h_n - u_n.  \]  Now the left hand side is holomorphic in
  $\Omega_1^-$, while the right hand side is holomorphic in $\Omega_2^- \cap \cdots \cap \Omega_n^-$,
  which contains the closure of $\Omega_1^+$.  Since the intersection of the domains on the left and the right
  hand sides is non-empty (it is in fact $\Sigma$), this shows that $h_1 - u_1$ has a holomorphic
  extension to $\sphere$ and therefore is a constant.  Since $h_1$ and $u_1$ both vanish at $\infty$,
  $h_1 = u_1$.  Continuing in this way, we can show that $h_i=u_i$ for all $i=1,\ldots, n$.
 \end{proof}

 Given this theorem, we can now prove the following.
 \begin{theorem}  \label{th:If_iso_general} Let $\Sigma$, $\Gamma_i$, $\Omega_i^+$, and $\Omega_i^-$ be as in Theorem \ref{th:decomposition}.
  Let $f=(f_1,\ldots,f_n)$ where $f_i:\disk^+ \rightarrow \Omega_i^+$ is a conformal map for $i=1,\ldots,n$.
  The map $\mathbf{I}_f$ is a bounded isomorphism.
 \end{theorem}
 \begin{proof}
  For any $h=(h_1,\ldots,h_n) \in \mathcal{D}_\infty(\disk^-)^n$, since $\Sigma \subset \Omega_i^-$ for all $i$, using Minkowski's inequality
  \begin{equation*}
    \| \mathbf{I}_f h \|_{\mathcal{D}_\infty(\Sigma)}  \leq
    \sum_{i=1}^n \| \mathrm{I}_{f_i} h_i \|_{\mathcal{D}_\infty(\Sigma)}  =
    \sum_{i=1}^n \left( \iint_{\Sigma} |(\mathrm{I}_{f_i} h_i)'|^2 \,dA \right)^{1/2}
     \leq \sum_{i=1}^n \| \mathrm{I}_{f_i} h_i \|_{\mathcal{D}_\infty(\Omega_i^-)}.
  \end{equation*}
  Since $\mathrm{I}_{f_i}$ is bounded for each $i$ by Theorem \ref{th:If_isomorphism_Grunsky_bounded}, this proves that
  $\mathbf{I}_f$ is bounded.  By Theorems \ref{th:If_isomorphism_Grunsky_bounded} and \ref{th:decomposition}, $\mathbf{I}_f$ is a bijection, so by
  the open mapping theorem, $\mathbf{I}_f$ is a bounded isomorphism.
 \end{proof}
 We conclude the section with two remarks to place the results in context.
 \begin{remark}
  In the statement of $\mathrm{Theorem}$ $\mathrm{\ref{th:decomposition}}$, the projection operators
  $P_\infty(\Omega_i^-)$ are defined using Cauchy integrals of the harmonic extension of $\left. h \right|_{\Gamma_i}$ to the {\it complement} $\Omega_i^+$ of $\Omega_i^-$.
  If we define the projection operator using integrals along curves inside $\riem \subset \Omega_i^-$, then the decomposition exists for much wider classes of functions, see e.g. \cite[Theorem 10.12]{DurenHpbook} for the case of Hardy spaces.

However, the definition of Faber operator uses a pull-back under conformal maps onto the complements $\Omega_i^+$.  Thus we need to use the fact that the projection using Cauchy integrals over curves approaching the boundary from the outside of $\riem$ has the same result as the projection defined using curves approaching from the outside.
 This makes the use of $\mathrm{Theorem}$ $\mathrm{\ref{th:all_about_projection}}$ unavoidable.  Furthermore, this theorem cannot be extended to more general function spaces such as Hardy spaces $($see the introduction$)$.
 \end{remark}
  \begin{remark}
   The proofs of $\mathrm{Theorems}$ $\mathrm{\ref{th:decomposition}}$ and $\mathrm{\ref{th:If_iso_general}}$ also require the fact that
  $\Gamma$ is a quasicircle.  This is because the statement and proof of $\mathrm{Theorem}$ $\mathrm{\ref{th:all_about_projection}}$ in turn requires the fact that any Dirichlet bounded function on $\Omega_i^-$ has a reflection, i.e. a Dirichlet-bounded harmonic function
  on $\Omega_i^+$ with the same boundary values.   That this is true for quasicircles is not entirely obvious, and was obtained in \cite[Theorem 3.5]{SchippersStaubach_jump}; recall also that for Dirichlet spaces this holds if and only if the curve is a quasicircle \cite{SchippersStaubach_Grunsky_quasicircle}.
 \end{remark}

 \end{subsection}
 \begin{subsection}{Multiple Faber series}
  As an application, we show that every function $h$ in $\mathcal{D}_\infty(\riem)$ has a unique multiple
  Faber series.  This polynomial approximation converges uniformly on compact subsets of $\riem$.
  This application also illustrates the meaning of the isomorphism $\mathbf{I}_f$.

  Let $\riem$, $\Omega^\pm$, and $f_k:\disk^+ \rightarrow \Omega_k^+$ be as in the previous section.
  There are many equivalent definitions of Faber polynomials, see \cite{Pommerenkebook,Suetin_monograph}.
  In our setting, for integers $m>0$ we define the $m$th Faber polynomial of $f_k$ by
   \begin{equation}
   \Phi^k_m(w) = -  \frac{1}{2 \pi i} \int_{f_k(\gamma_r)} \frac{(f_k^{-1}(\zeta))^{-m}}
   {\zeta-z} \,d\zeta
  \end{equation}
  where $\gamma_r$ is the circle $|\zeta|=r$ traced counterclockwise for some fixed $r \in (0,1)$ such that
  $z$ is in the unbounded component of the complement of $f(\gamma_r)$.
  
  Let $p_i=f(0)$.  $\Phi^k_m(z)$ is a polynomial of degree $m$ in $1/(z-p_i)$.  To see this, expand $f_k^{-1}(\zeta)^{-m}$ in a Laurent series
  \begin{equation} \label{eq:fminusm_expansion}
    f_k^{-1}(\zeta)^{-m} = \frac{c_{-m}}{(\zeta - p_i)^m} + \cdots \frac{c_1}{\zeta - p_i} + c_0 + c_{-1} (\zeta - p_i) + \cdots
  \end{equation}
  in  $0< |\zeta -p_i| < \delta$ for $\delta>0$ sufficiently small, and deform the curve $f_k(\gamma_r)$
  to $|\zeta - p_i| = \delta/2$.  We then see that $\Phi_m^k$ is the principle part of (\ref{eq:fminusm_expansion}).  
  By \cite[Theorem 3.4]{SchippersStaubach_Grunsky_quasicircle}, we have
  \[   \Phi^k_m = I_{f_k} (z^{-m})  \]
  (note that by $z^{-m}$ we mean the function $z \mapsto z^{-m} \in \mathcal{D}_\infty(\disk^-)$).

   We define a multiple Faber series of a holomorphic function $h$ on $\riem$ to be a series of the form
   \[  h(z) = \sum_{m=1}^\infty \sum_{k=1}^n a^k_m \Phi^k_m(z)  \]
   for complex constants $a^k_m$, converging uniformly on compact subsets of $\riem$.  We then have the following.
   \begin{theorem}  Let $\Sigma$ be a multiply-connected domain in $\sphere$ containing $\infty$ bounded
  by non-intersecting quasicircles $\Gamma_i$, $i=1,\ldots,n$.  Let $\Omega_i^+$ and $\Omega_i^-$ be the bounded and unbounded components
  respectively of $\Gamma_i$.  Let $f_i:\disk^+ \rightarrow \Omega_i^+$ be conformal maps for $i=1,\ldots,n$. Then every $h \in \mathcal{D}(\Sigma)$ has a unique Faber series expansion
  \[  h(z) = \sum_{m=1}^\infty \sum_{k=1}^n a^k_m \Phi^k_m(z)  \]
  which converges uniformly on compact subsets of $\riem$.
  The coefficients $a^k_m$ are the coefficients of the expansion
   \[   g_k(z)= \sum_{m=1}^\infty a^k_m z^{-m}             \]
   where
   \[  (g_1,\ldots,g_n) = \mathbf{I}_f^{-1}(h).   \]
   \end{theorem}
   \begin{proof}
    Let $h= h_1 + \cdots h_n$ be the decomposition obtained in Theorem \ref{th:decomposition}.
    Applying \cite[Theorem 3.17]{SchippersStaubach_Grunsky_quasicircle} to each $h_k \in \mathcal{D}(\Omega_k^-)$, we see that each $h_k$ has a unique Faber series
    \[  h_k(z)= \sum_{m=1}^\infty a^k_m \Phi^k_m(z)   \]
    which converges uniformly on compact subsets of $\Omega_i^+$.
    Setting $g_k(z) = \sum_{m=1}^\infty a^k_m z^{-m}$ we have $I_{f_k} (g_k) = h_k$ and
    $\mathbf{I}_f (g_1,\ldots,g_n) = h_1 + \cdots + h_n = h$.  Since $\mathbf{I}_f$ is an
    isomorphism by Theorem \ref{th:If_iso_general}, this proves the claim.
   \end{proof}
 \end{subsection}
\end{section}

\begin{section}{Dirichlet spaces as graphs of Grunsky operators}  \label{se:Dirichlet_Grunsky}
\begin{subsection}{The Dirichlet space of a multiply-connected domain as the graph of a Grunsky operator}
 In this section we describe the Dirichlet space of $\Sigma$ as the graph of a certain generalized
 Grunsky operator.  We first show this for $\riem$ simply-connected, and then extend to the general
 case.

Let $\Gamma$ be a quasicircle not containing $\infty$,
 and let $\Omega^+$ and $\Omega^-$ be the bounded and unbounded components of the complement
 respectively.  Let $f:\disk^+ \rightarrow \Omega^+$ be a fixed conformal map.
 Any element $h \in \mathcal{H}(\Gamma)$ satisfies $\mathcal{C}_f h \in \mathcal{H}(\mathbb{S}^1)$
 and therefore has a Fourier series
 \[  h \circ f(e^{i\theta}) = \sum_{n=-\infty}^\infty a_n e^{in \theta} \]
 which converges to $h \circ f$ almost everywhere (in fact, except on a set of logarithmic capacity zero).
 We will ignore the constant terms.  Define
 \[  \mathcal{H}_0(\mathbb{S}^1) = \left\{ h(e^{i\theta}) = \sum_{n=-\infty, n \neq 0}^\infty a_n e^{i n \theta} \in \mathcal{H}(\mathbb{S}^1)  \right\}.  \]
Note that this has the decomposition $\mathcal{H}_0(\mathbb{S}^1) = \mathcal{D}_\infty(\disk^-) \oplus
 \mathcal{D}_0(\disk^+)$ obtained by replacing $e^{i\theta}$ by $z$.

 Define
 \[  \hat{\mathcal{C}}_f h = h \circ f - {h \circ f}(0).  \]
 Note that with this definition, we have $\mathrm{I}_f= \mathrm{P}_\infty(\Omega^-) \hat{\mathcal{C}}_{f^{-1}}$
 and $\Gr_f = \mathrm{P}_\infty(\Omega^-) \hat{\mathcal{C}}_f \mathrm{I}_f$. We will use these facts without
 comment ahead.\\

 We pose the following question:\\ 

 {\bf Question (1)}:
  Which Fourier series in $\mathcal{H}_0(\mathbb{S}^1)$ arise as $\hat{\mathcal{C}}_f h$
 for some $h \in \mathcal{D}_\infty(\Omega^-)$?
 That is, how can one characterize the set
  \[  \mathcal{W} = \hat{\mathcal{C}}_f \mathcal{D}_{\infty}(\Omega^-)?   \]
 The answer to this question is a characterization of the set of boundary
 values of $\mathcal{D}(\Omega^-)$.  
 
 The following theorem answers this question.
 \begin{theorem} \label{th:Grunsky_graph_simply} Let $\Gamma$ be a quasicircle not containing $\infty$, and let $\Omega^+$
  and $\Omega^-$ be the bounded and unbounded components respectively.  Let $f:\disk^+ \rightarrow
  \Omega^+$ be a conformal map.  Then
  \[  \mathrm{P}_\infty(\disk^-) \hat{\mathcal{C}}_f \mathrm{I}_f = \mathrm{Id}  \]
  where $\mathrm{Id}:\mathcal{D}_\infty(\disk^-) \rightarrow \mathcal{D}_\infty(\disk^-)$ is the identity on
  $\mathcal{D}(\disk^-)$, and
  \[  \mathrm{P}_0(\disk^+) \hat{\mathcal{C}}_f \mathrm{I}_f = \Gr_f.      \]
  Thus $\mathcal{W}$ is the graph of $\Gr_f$ in
  $\mathcal{H}_0(\mathbb{S}^1) =  \mathcal{D}_\infty(\disk^-) \oplus \mathcal{D}_0(\disk^+)$.
 \end{theorem}
 \begin{proof}
  The first claim is \cite[Equation 3.7]{SchippersStaubach_Grunsky_quasicircle}.
  The second claim is the definition of the Grunsky operator.  The fact
  that $\mathcal{W}$ is the graph of $\Gr_f$ follows immediately.
 \end{proof}

  This also shows that for any $h \in \mathcal{W}$, the element $H \in \mathcal{D}(\Omega^-)$
  such that $\hat{\mathcal{C}}_f H = h$ is uniquely determined and is given by
  $H = \mathrm{I}_f \mathrm{P}_\infty(\disk^-) h$.

 \begin{remark}  The constant term can be included if desired, by replacing
  $\mathrm{P}_\infty(\disk^-)$ with the Cauchy integral
  (which does not annihilate the constant term), and altering the definition
  of the Grunsky operator accordingly.  The matrix of the new Grunsky operator would then contain the so-called
  logarithmic coefficients of $f$. Although the constant term is significant,
  it does not play an important role in this paper, and furthermore introduces some   notational complications
  in the multiply-connected case.  We omit the treatment here.
 \end{remark}

 We shall now proceed by defining a generalized Grunsky operator.  First we define block entries of the operator
 and show that they are bounded.
 \begin{theorem} \label{th:Grunsky_block_bounded}
  Let $\Sigma$ be a multiply-connected domain bounded by $n$ quasicircles
 $\Gamma_i$, let $\Omega_i^+$ and
 $\Omega_i^-$ be the bounded and unbounded components respectively.   Let $f=(f_1,\ldots,f_n)$
 where  $f_i:\disk^+ \rightarrow \Omega_i^+$ are conformal maps.  The operators
 \begin{align}  \label{eq:Grunsky_nonoverlap_defn}
  \Gr_{ji}(f): \mathcal{D}_\infty (\disk^-)&  \longrightarrow \mathcal{D}_0(\disk^+) \nonumber \\
  h & \longmapsto \mathrm{P}_0(\disk^+) \hat{\mathcal{C}}_{f_j} \mathrm{I}_{f_i} h.
 \end{align}
 are bounded for all $i,j=1,\ldots,n$.
 \end{theorem}
 \begin{proof}
  By definition $\Gr_{ii}(f)=\Gr_{f_i}$.  Thus by Theorem \ref{th:If_isomorphism_Grunsky_bounded}
  it is bounded.  Now assume that $i \neq j$.  Since $\mathrm{P}(\disk^+)$ and $\mathrm{I}_{f_i}$ are bounded,
  it is enough to show that $\hat{\mathcal{C}}_{f_j}$ is a bounded operator on $\mathcal{D}_\infty(\Omega_i^-)$.
  Since $\Omega_j^+ \subseteq \Omega_i^-$ when $i \neq j$, for any $h \in \mathcal{D}_\infty(\Omega_i^-)$ we have
  that $\hat{\mathcal{C}}_{f_j} h  \in \mathcal{D}_0(\disk^+)$ and by a change of variables
  (setting $p_j = f_j(0)$)
  \[  \| \hat{C}_{f_j} h \|_{\mathcal{D}_0(\disk^+)} = \| h \|_{\mathcal{D}_{p_j}(\Omega_j^+)}
   \leq \| h \|_{\mathcal{D}_\infty(\Omega^-_i)}.  \]
 \end{proof}
 \begin{remark}  \label{re:projection_extraneous} Since $f_j(\disk^+) \subseteq \Omega_i^-$
 whenever $i \neq j$,
  the projection $P_0(\disk^+)$ can be removed when $i \neq j$.
 \end{remark}

 \begin{theorem}  \label{th:Grunsky_off_diagonal}
  Let $\Gamma_i$, $\Omega_i^+$, $\Omega_i^-$ and $f_i$ be as in Theorem \ref{th:Grunsky_block_bounded}
  for $i=1,\ldots,n$.     If $i \neq j$, then for $h \in \mathcal{D}_\infty(\disk^-)$,
  $\Gr_{ji}(f)$ has the integral expression
  \[ \left[ \Gr_{ji}(f) \right]h(z) = \frac{1}{\pi} \iint_{\disk^+} \left(
  \frac{f_i'(\zeta)}{f_i(\zeta)-f_j(z)} - \frac{f_i'(\zeta)}{f_i(\zeta) - f_j(0)} \right)
\mathbf{R}h(\zeta)  \, dA(\zeta).  \]
 \end{theorem}
 Observe that the integral kernel is non-singular.
 \begin{proof}
  Using equation (\ref{eq:If_alternate_using_r}) we have that up to the constant term
  the integral expression for $\Gr_{ji}(f) h$ is
  \begin{align*}
    - \lim_{r \nearrow 1} \int_{f_i (\gamma_r)}
   \frac{\left( \mathbf{R} h \right) \circ f_i^{-1}(\zeta)}{\zeta - f_j(z)} d\zeta
   & = - \lim_{r \nearrow 1} \int_{\gamma_r}
   \frac{\mathbf{R}h(\zeta)}{f_i(\zeta) - f_j(z)} f_i'(\zeta) d\zeta \\
   & = \frac{1}{\pi} \iint_{\disk^+} \frac{\mathbf{R}h(\zeta)}{f_i(\zeta)-f_j(z)} f_i'(\zeta) dA(\zeta).
  \end{align*}
  The claim follows by subtracting the value of this integral at $0$.
 \end{proof}

 Generalized Grunsky operators for pairs of non-overlapping maps were considered by Hummel \cite{Hummel},
 who derived Grunsky inequalities in matrix form for such maps (the function class was
 formulated differently but can be seen to be the same by applying a transform and
 an identity for the Grunsky matrix).  They were also considered in Takhtajan
 and Teo \cite{TakhtajanTeo_memoirs} in the equivalent $L^2$ setting,  when
 the complement of the maps is of measure zero.

 We can now define a generalized Grunsky operator.  Let
 $\mathcal{D}_0(\disk^+)^n = \mathcal{D}_0(\disk^+) \oplus \cdots \oplus \mathcal{D}_0(\disk^+)$
 where there are $n$ summands.
 \begin{definition}
   For $\Sigma$, $\Omega_i^+$, $\Omega_i^-$ and $f=(f_1,\ldots,f_n)$
 as above, the generalized Grunsky operator is
 \begin{align*}
  \Grfull(f) : \mathcal{D}_\infty(\disk^-)^n & \longrightarrow \mathcal{D}_0(\disk^+)^n \\
  (H_1,\ldots,H_n) & \longmapsto \left( \sum_{i=1}^n \Gr_{1i} H_i , \ldots, \sum_{i=1}^n \Gr_{ni} H_i \right).
 \end{align*}
 \end{definition}
 It follows immediately from Theorem \ref{th:Grunsky_block_bounded} that $\Grfull(f)$ is bounded.\\
 
The following result generalizes the classical result of C. Pommerenke \cite{Pommerenkebook}
 for quasicircles in one direction.
 \begin{theorem} \label{th:full_Grunsky_bounded_by_one}  With $\Sigma$, $\Omega_i^+$, $\Omega_i^-$ and $f=(f_1,\ldots,f_n)$
  as in Theorem $\ref{th:Grunsky_block_bounded}$, $$\| \Grfull(f) \|_{\mathcal{D}_{\infty}(\disk^-)^n\to\mathcal{D}_{0}(\disk^+)^n} < 1.$$
 \end{theorem}
 \begin{proof}
   Let $g_i:\disk^- \rightarrow \Omega_i^-$ and $f_i:\disk^+ \rightarrow \Omega_i^+$ be conformal maps; assume that $g_i(\infty)=\infty$.
   Let $H=(H_1,\ldots,H_n) \in \mathcal{D}_\infty(\disk^-)^n$.  First, a standard
   computation using Green's identity (see the proof of \cite[Theorem 3.10]{SchippersStaubach_Grunsky_quasicircle}
   where it appears with notation similar to that here)
   shows that
   \begin{equation} \label{eq:Greens_identity_temp}
    \| \mathrm{I}_{f_i} H_i \|^2_{\mathcal{D}(\Omega_i^-)}
     = -\lim_{r \searrow 1} \frac{1}{2 \pi i} \int_{g_i(\gamma_r)} H_i'(z) \overline{H_i(z)} dz
      = \|H_i \|_{\mathcal{D}_\infty(\disk^-)}^2 - \| \Gr(f)_{ii} H_i \|^2_{\mathcal{D}_0(\disk^+)}.
   \end{equation}
   Thus
   since quasicircles have Lebesgue measure zero \cite{Lehto}, for any fixed $j$
   \begin{align*}
    \| \mathrm{I}_{f_j} H_j \|_{\mathcal{D}(\Omega_j^-)}^2 & = \| \mathrm{I}_{f_j} H_j \|_{\mathcal{D}(\riem)}^2
      + \sum_{i=1, i \neq j}^n \| \mathrm{I}_{f_j} H_j \|_{\mathcal{D}(f_i(\disk^+))}^2 \\
      & =  \| \mathrm{I}_{f_j} H_j \|_{\mathcal{D}(\riem)}^2
      + \sum_{i=1, i \neq j}^n \| \Gr(f)_{ij} H_j \|_{\mathcal{D}_0(\disk^+)}^2
   \end{align*}
   where in the last step we use a change of variables and Remark \ref{re:projection_extraneous}.
   Combining this with (\ref{eq:Greens_identity_temp}) we see that
   \begin{equation} \label{eq:Grunsky_bound_proof_temp}
    \sum_{i=1}^n \| \Gr(f)_{ij} H_j \|^2_{\mathcal{D}_0(\disk^+)}
     = \| H_j \|^2_{\mathcal{D}_\infty(\disk^-)} -
     \|I_{f_j} H_j \|_{\mathcal{D}_\infty(\riem)}^2.
   \end{equation}

   Since $\mathbf{I}_{f}$ is an isomorphism, there is a $c  \in (0,1)$ such that
   \[  \| \mathbf{I}_f (H_1,\ldots,H_n) \|^2_{\mathcal{D}_\infty(\riem)}
      = \left\| \sum_{j=1}^n \mathrm{I}_{f_j} H_j \right\|^2_{\mathcal{D}_\infty(\riem)}
        \geq (1-c) \sum_{j=1}^n \| H_j\|^2_{\mathcal{D}_\infty(\disk^-)}  \]
   for all $H_1,\ldots,H_n \in \mathcal{D}_\infty(\disk^-)$.  Thus by Minkowski's inequality and the fact that
 $(\sum_{j=1}^{n} |a_j|)^2 \leq n \sum_{j=1}^{n} |a_j|^2,$ one has
   \[  \sum_{j=1}^n \| I_{f_j} H_j \|^2_{\mathcal{D}_\infty(\riem)} \geq \frac{1-c}{n}
     \sum_{j=1}^n \|H_j\|^2_{\mathcal{D}_\infty(\disk^-)}.  \]
   Combining this with (\ref{eq:Grunsky_bound_proof_temp}) we obtain
   \[  \sum_{j=1}^n \sum_{i=1}^n \| \Gr(f)_{ij} H_j \|^2_{\mathcal{D}_0(\disk^+)}\leq
      \frac{n-1+c}{n}\sum_{j=1}^n  \|H_j\|^2_{\mathcal{D}_\infty(\disk^-)}\]
   which completes the proof, since $(n-1+c)/n <1$, for all $n\geq 1$ and all $c\in (0,1)$.
 \end{proof}

 Now define
 \begin{align*}
  \hat{\mathcal{C}}_f: \mathcal{D}_\infty(\Sigma) & \longrightarrow \mathcal{H}_0(\mathbb{S}^1)^n \\
  H & \longmapsto \left( \hat{\mathcal{C}}_{f_1} H, \ldots,\hat{\mathcal{C}}_{f_n} H \right)
 \end{align*}
 and
 \[  \mathbf{P}(\disk^+)  = \mathrm{P}_0(\disk^+) \oplus \cdots \oplus \mathrm{P}_0(\disk^+) \]
 and similarly
 \[  \mathbf{P}(\disk^-)  = \mathrm{P}_\infty(\disk^-) \oplus \cdots \oplus \mathrm{P}_\infty(\disk^-).  \]
 Finally let
 \[  \mathbf{W} = \hat{\mathcal{C}}_f \mathcal{D}_\infty(\Sigma)
   = \{ (h_1,\ldots,h_n) \in \mathcal{H}_0(\mathbb{S}^1)^n \,:\, h_i = \hat{\mathcal{C}}_{f_i} H \ \text{for some} \ H \in \mathcal{D}_\infty(\riem) \}.  \]

We now ask the following question:\\
  
  {\bf Question (2)}:
  Which Fourier series in $\bigoplus^n \mathcal{H}_0(\mathbb{S}^1)$ arise as $\hat{\mathcal{C}}_f h$
 for some $h \in \mathcal{D}_\infty(\riem)$?
 That is, how can one characterize the set
  \[  \mathbf{W} = \hat{\mathcal{C}}_f \mathcal{D}_{\infty}(\riem)?   \]

 We now show that $\mathbf{W}$ is the graph of the generalized Grunsky operator.
 \begin{theorem}  Let $\Sigma$, $\Omega_i^+$, $\Omega_i^-$ and $f=(f_1,\ldots,f_n)$ be as in
  Theorem $\ref{th:Grunsky_block_bounded}$.  We have that
  \[  \mathbf{P}(\disk^-) \hat{\mathcal{C}}_f \mathbf{I}_f = \mathbf{Id}  \]
  where $\mathbf{Id}:\mathcal{D}_\infty(\disk^-)^n \mapsto \mathcal{D}_\infty(\disk^-)^n$
  is the identity map and
  \[  \mathbf{P}(\disk^+) \hat{\mathcal{C}}_f \mathbf{I}_f = \Grfull(f).  \]
  Thus $\mathbf{W}$ is the graph of $\Grfull(f)$.
 \end{theorem}
 \begin{proof}
  As observed in the proof of Theorem \ref{th:Grunsky_block_bounded}, if $i \neq j$ then
  $\Omega_j^+ \subseteq \Omega_i^-$.  Thus for all $h \in \mathcal{D}_\infty(\disk^-)$,
   $\hat{\mathcal{C}}_{f_j} \mathrm{I}_{f_i} h \in
  \mathcal{D}_0(\disk^+)$, and  so $\mathrm{P}_\infty(\disk^-) \hat{\mathcal{C}}_{f_j}
   \mathrm{I}_{f_i} h =0$.   Thus to prove the first claim it suffices to show that
   $\mathrm{P}_\infty(\disk^-) \hat{\mathcal{C}}_{f_i} \mathrm{I}_{f_i}$ is the identity
   on $\mathcal{D}_\infty(\disk^-)$.  This follows from Theorem \ref{th:Grunsky_graph_simply}, and the first claim is thereby proven.
The second claim is the definition of $\Grfull(f)$.
 \end{proof}
\end{subsection}
\end{section}

\end{document}